\documentclass[11pt,twoside]{article} 
\usepackage{bbm}
\usepackage{amsmath,amssymb}
\usepackage{slashed}
\makeatletter
\newcommand*{\mint}[1]{%
  \mint@l{#1}{}%
}
\newcommand*{\mint@l}[2]{%
  \@ifnextchar\limits{%
    \mint@l{#1}%
  }{%
    \@ifnextchar\nolimits{%
      \mint@l{#1}%
    }{%
      \@ifnextchar\displaylimits{%
        \mint@l{#1}%
      }{%
        \mint@s{#2}{#1}%
      }%
    }%
  }%
}
\newcommand*{\mint@s}[2]{%
  \@ifnextchar_{%
    \mint@sub{#1}{#2}%
  }{%
    \@ifnextchar^{%
      \mint@sup{#1}{#2}%
    }{%
      \mint@{#1}{#2}{}{}%
    }%
  }%
}
\def\mint@sub#1#2_#3{%
  \@ifnextchar^{%
    \mint@sub@sup{#1}{#2}{#3}%
  }{%
    \mint@{#1}{#2}{#3}{}%
  }%
}
\def\mint@sup#1#2^#3{%
  \@ifnextchar_{%
    \mint@sup@sub{#1}{#2}{#3}%
  }{%
    \mint@{#1}{#2}{}{#3}%
  }%
}
\def\mint@sub@sup#1#2#3^#4{%
  \mint@{#1}{#2}{#3}{#4}%
}
\def\mint@sup@sub#1#2#3_#4{%
  \mint@{#1}{#2}{#4}{#3}%
}
\newcommand*{\mint@}[4]{%
  \mathop{}%
  \mkern-\thinmuskip
  \mathchoice{%
    \mint@@{#1}{#2}{#3}{#4}%
        \displaystyle\textstyle\scriptstyle
  }{%
    \mint@@{#1}{#2}{#3}{#4}%
        \textstyle\scriptstyle\scriptstyle
  }{%
    \mint@@{#1}{#2}{#3}{#4}%
        \scriptstyle\scriptscriptstyle\scriptscriptstyle
  }{%
    \mint@@{#1}{#2}{#3}{#4}%
        \scriptscriptstyle\scriptscriptstyle\scriptscriptstyle
  }%
  \mkern-\thinmuskip
  \int#1%
  \ifx\\#3\\\else_{#3}\fi
  \ifx\\#4\\\else^{#4}\fi
}
\newcommand*{\mint@@}[7]{%
  \begingroup
    \sbox0{$#5\int\m@th$}%
    \sbox2{$#5\int_{}\m@th$}%
    \dimen2=\wd0 %
    \let\mint@limits=#1\relax
    \ifx\mint@limits\relax
      \sbox4{$#5\int_{\kern1sp}^{\kern1sp}\m@th$}%
      \ifdim\wd4>\wd2 %
        \let\mint@limits=\nolimits
      \else
        \let\mint@limits=\limits
      \fi
    \fi
    \ifx\mint@limits\displaylimits
      \ifx#5\displaystyle
        \let\mint@limits=\limits
      \fi
    \fi
    \ifx\mint@limits\limits
      \sbox0{$#7#3\m@th$}%
      \sbox2{$#7#4\m@th$}%
      \ifdim\wd0>\dimen2 %
        \dimen2=\wd0 %
      \fi
      \ifdim\wd2>\dimen2 %
        \dimen2=\wd2 %
      \fi
    \fi
    \rlap{%
      $#5%
        \vcenter{%
          \hbox to\dimen2{%
            \hss
            $#6{#2}\m@th$%
            \hss
          }%
        }%
      $%
    }%
  \endgroup
}
\usepackage{mathrsfs}
\usepackage{amssymb}
\usepackage{amsmath}
\usepackage{amsthm}
\usepackage{amsfonts}
\usepackage{color}
\usepackage{graphicx}
\usepackage[active]{srcltx}
\usepackage{tikz}
\usepackage[cp1252]{inputenc}

\usepackage{mathrsfs}
\usepackage{graphicx}

\usepackage[active]{srcltx}

\allowdisplaybreaks

\usepackage{titletoc}
\titlecontents{section}[0pt]{\addvspace{2pt}\filright}
              {\contentspush{\thecontentslabel\ }}
              {}{\titlerule*[8pt]{.}\contentspage}

\def\rr{{\mathbb R}}
\def\rn{{{\rr}^n}}

\def\nn{{\mathbb N}}

\def\fz{\infty}
\def\az{\alpha}

\def\loc{{\mathop\mathrm{\,loc\,}}}

\def\dz{\delta}
\def\bdz{\Delta}
\def\ez{\epsilon}

\def\kz{\kappa}
\def\bz{\beta}

\def\gz{{\gamma}}

\def\pa{\partial}

\def\nb{\nabla}

\def\bint{{\ifinner\rlap{\bf\kern.35em--}
\int\else\rlap{\bf\kern.45em--}\int\fi}\ignorespaces}

\def\bbint{{\ifinner\rlap{\bf\kern.35em--}
\hspace{0.078cm}\int\else\rlap{\bf\kern.45em--}\int\fi}\ignorespaces}

\def\ddiv{{\rm div}}

\def\lr{\right}
\def\lf{\left}
\def\la{\langle}
\def\ra{\rangle}

\newtheorem{thm}{Theorem}[section]
\newtheorem{lem}[thm]{Lemma}
\newtheorem{rem}[thm]{Remark}
\newtheorem{cor}[thm]{Corollary}
\numberwithin{equation}{section}

\title
{\Large\bf   A quantitative second order estimate for (weighted) $p$-harmonic functions in    manifolds under curvature-dimension condition
\footnotetext{\hspace{-0.35cm}
\endgraf
2020 {\it Mathematics Subject Classification:} 58J05 $\cdot$ 35J92
\endgraf The first author is funded by the Academy of Finland (Grant No. 321896 and Grant No. 328846).
The second author is funded by NSFC (Grant  No. 12171018).
The third author is funded by NSFC (Grant No. 12025102).
 \endgraf Data availability: No data was used for the research described in the article.
}}

\author{Jiayin Liu, Shijin Zhang and Yuan Zhou}
 \date{}
\begin{document}

\arraycolsep=1pt
\allowdisplaybreaks
 \maketitle

\begin{center}
\begin{minipage}{13.5cm}\small
 \noindent{\bf Abstract.}\quad
We build up a quantitative {\color{black}second-order} Sobolev estimate  of $ \ln w$
for positive $p$-harmonic functions $w$ in Riemannian manifolds under Ricci curvature bounded from {\color{black}below} and also for positive weighted $p$-harmonic functions $w$
  in  weighted manifolds under the Bakry-\'{E}mery  curvature-dimension condition.

\end{minipage}
\end{center}


\section{Introduction}

Let $(M^n,g)$ be a complete non-compact Riemannian manifold with dimension $n\ge2$.
Suppose that the Ricci curvature is bounded from below,  that is,
$   Ric_g  \ge - \kz$    for some $\kz\ge0$. For any positive
harmonic function $w$ in a domain $\Omega\subset M^n$,   Cheng-Yau  \cite{cy}    established the
 following  famous gradient estimate:
\begin{equation}\label{cyg}|\nb\ln w|=\frac{|\nb w|}{w} \le C(n)\frac{1+\sqrt{\kz} r}{r} \quad \text{in $    B(z,r)\subset B(z,2r)\subset\Omega$}.
\end{equation}
Recall that a  harmonic function  $w$ in  $\Omega $ is a weak solution   to the Laplace equation
$$\Delta w:={\rm div}(\nb w)=0\mbox{ in $\Omega$.}$$
We  also refer to   \cite[Theorem 1.3]{nst} for a quantitative $ W^{2,2}_\loc$-regularity
of harmonic functions.

Motivated by  the application  in
 the inverse mean curvature flow (see \cite{kn09,mrs}),  Cheng-Yau type gradient estimate was  extended by   \cite{m07,kn09,wz10,mrs} to  $p$-harmonic functions   in $\Omega$ for $1<p<\fz$, that is,
 weak solutions  to the $p$-Laplace equation
$$\bdz_p w ={\rm div}(|\nb w|^{p-2}\nb w)=0 \ \mbox{
 in  $\Omega$. }$$
Precisely, if $(M^n,g)$ is flat (that is, {\color{black}the} Euclidean space $\mathbb R^n$) or its sectional curvature is bounded from below by $-\kz$,
via  Cheng-Yau's approach  Moser \cite{m07} and Kotschwar-Ni \cite{kn09} showed that any positive $p$-harmonic function $w $ in $\Omega$ satisfies
\begin{equation}\label{p-ly}
   |\nb\ln w| \le C(n )\frac{1+\sqrt{\kz} r}{r} \quad \mbox{in $      B(z,r)\subset B(z,2r)\subset\Omega$},
   \end{equation}
where  the constant $C(n)>0$  is  independent  of $p\in(1,\fz)$.
 Under the Ricci curvature lower bound $   Ric_g  \ge - \kz$, it was asked in \cite{kn09} whether
\eqref{p-ly} holds or not.   Some progress was made as below.
Based on Cheng-Yau's argument, Wang-Zhang \cite{wz10} proved that
 \begin{equation}\label{secor} \mbox{$|\nb \ln w|^{\frac{p-\gz}{2}}\in W^{1,2}_\loc $ with  ${\color{black}\gz < 0}$}
\end{equation}
and
the following weaker revision of \eqref{p-ly}:
\begin{equation}\label{p-ly2}
   |\nb\ln w| \le C(n,p )\frac{1+\sqrt{\kz} r}{r} \quad \mbox{in $    B(z,r)\subset B(z,2r)\subset\Omega$},
   \end{equation}
 where  the constant $C(n,p)>0$  blows up as $p\to1$.
Recently, with the aid of the fake distance coming from capacity,
 $ C(n,p)$  was  proved by Mari-Rigoli-Setti \cite{mrs} to be bounded by  $\frac{n-1}{p-1}$ as $p\to1$.
Moreover, \eqref{secor} and \eqref{p-ly2}  were generalized to   weighted manifolds  $(M^n,g,e^{-h}d{\rm  vol}_g)$.
A weighted $p$-harmonic function $w$ in a domain $\Omega\subset M^n$ is a  weak solution  to the weighted $p$-harmonic equation $$\Delta_{p,h}w:=e^{h}{\rm div}(e^{-h}|\nb w|^{p-2}\nb w)=0 \ \mbox{in $\Omega$}.$$
Under the Bakry-\'{E}mery  curvature-dimension  condition  $  Ric_h^N\ge -\kz$ for some $N\in[n,\fz) $ and $\kz\ge 0$ (see Section 2 for details),
 Dung-Dat \cite{dd16} showed that 
if $w>0$, then  $|\nb \ln w|^{\frac{p-\gz}{2}}\in W^{1,2}_\loc $ with  $\gz<0$  and also  \begin{align} \label{mest0}
 |\nb \ln w| \le  C(n,N,p)  \frac{ 1+\sqrt{\kz} r }{r } \quad \text{in $    B(z,r)\subset B(z,2r)\subset\Omega$.}
\end{align}

The main aim of this paper is to
build up a quantitative second-order Sobolev estimate  of $ \ln w$
for positive $p$-harmonic functions $w$ in Riemannian manifolds under Ricci curvature
bounded from {\color{black}below} and also for positive weighted $p$-harmonic functions $w$
  in  weighted manifolds under the Bakry-\'{E}mery  curvature-dimension condition. See Theorem 1.1 and Theorem 1.2 separately.
These improve  the corresponding  second-order Sobolev regularity  
in \cite{wz10,dd16}  mentioned  above.

 To be precise, under the Ricci curvature lower bound, we have the following result.
For convenience, below we write $\bbint_{E}f\,dm$ as the average of $f$ in the set $E$ with respect to the measure $m$, that is, $\bbint_{E}f\,dm=\frac1{m(E)}\int_Ef\,dm$. We use {\color{black}$C(a_1, \cdots, a_m)$ to denote a positive constant depending on absolute constants $a_1, \cdots, a_m$}.
\begin{thm} \label{thm0}
 Suppose that $(M^n,g)$ satisfies $Ric_g \ge
 -  \kz$ for some $\kz\ge 0$.
Let $1<p<\fz$ and $\gz<3+ \frac{p-1}{n-1}$.
For any positive $p$-harmonic function {\color{black}w} in a domain $\Omega\subset M$,
we have   $|\nb \ln w|^{\frac{p-\gz}{2}}\nb \ln w\in W^{1,2}_\loc(\Omega)$ and
\begin{align} \label{mest02}
\bint_{B(z,r) } \lf| \nb[|\nb \ln w|^{\frac{p-\gz}{2}}\nb \ln w ]\lr |^2\, d{\rm vol}_g
& \le   C(n,p,\gz)  \left[\frac{1+\sqrt{\kz} r}{r}\right]^{p-\gz+4}  e^{\sqrt{\kz}r}
\end{align}
 whenever $  B(z,4r)\Subset \Omega$.

In particular, if $1<p<3+\frac{2}{n-2}$, then
   $ \nb^2\ln w\in L^{2}_\loc(\Omega)$   and
\begin{align} \label{mest01}
\bint_{B(z,r)}|\nb^2 \ln w|^2\, d{\rm vol}_g
& \le   C(n,p)  \left[\frac{ 1+\sqrt{\kz} r }{r}\right]^4  e^{\sqrt{\kz}r}
\end{align}
whenever $  B(z,4r)\Subset \Omega$.
\end{thm}
Here and throughout the paper for domains $A$ and $B$, the notation $A\Subset B$ stands for that $A$
is a bounded subdomain of $B$ and its closure $A \subset B$.

Recall that if $(M^n,g)$ is flat, that is, the {\color{black}Euclidean} space $\rn$,   $ p$-harmonic functions $w$  in a domain $\Omega\subset \rn$ are proved to satisfy
 $|\nabla w|^{\frac{p-\gz}2}\nabla w\in W^{1,2}_\loc(\Omega)$ with some quantative bound
whenever   $\gz<3+\frac{p-1}{n-1}$
  see    \cite{mw88,im89,dpzz,ss20} and also the {\color{black}references} therein for some earlier partial results.
In particular,  if $1<p<3+\frac2{n-2}$,    noting $p<3+\frac{p-1}{n-1}$ and taking $\gz=p$, one has  $w\in W^{2,2}_\loc(\Omega)$.
When $n\ge3$ and $p\ge 3+\frac2{n-2}$,  it is not clear whether
  $w\in W^{2,2}_\loc(\Omega)$  or not.  When $n=2$, the range $\gz< 3+\frac{p-1}{n-1}=p+2$ is optimal as  {\color{black}witnessed} by some construction in \cite{im89}.

Moreover, we extend Theorem \ref{thm0} to weighted {\color{black}manifolds}   satisfying Bakry-\'{E}mery curvature-dimension  condition,

 \begin{thm}\label{thm2} Let $(M^n,g,e^{-h}{\rm vol}_g)$ be a
weighted manifold with  $Ric_h^N\ge -\kz$ for some $n\le N<\fz$ and $\kz\ge0$.
Let
$1<p<\fz$ and $\gz<3+ \frac{p-1}{N-1}$.
For any positive weighted $p$-harmonic function $w$ in a domain $\Omega\subset M$, we have
 $|\nb \ln w|^{\frac{p-\gz}{2}}\nb \ln w\in W^{1,2}_\loc(\Omega)$ and
\begin{align} \label{mest2}
\bint_{B(z,r)}\lf| \nb [|\nb \ln w|^{\frac{p-\gz}{2}}\nb \ln w ]\lr |^2\, d{\rm vol}_h
& \le   C(n,N,p,\gz) \left[\frac{1+\sqrt \kz r}r\right]^{p-\gz+4} 
e^{\sqrt{\kz}r}\
\end{align}
 whenever $  B(z,4r)\Subset \Omega$.

 In particular, if $p\in(1,3+\frac{2}{N-2})$, then
   $ \nb^2\ln w\in L^{2}_\loc(\Omega)$   and
\begin{align} \label{mest1}
\bint_{B(z,r)}|\nb^2 \ln w|^2\,  d{\rm vol}_h
& \le   C(n, N,p)  \left[\frac{ 1+\sqrt{\kz} r }{r }\right]^4  e^{\sqrt{\kz}r}
\end{align}
whenever  $  B(z,4r)\Subset \Omega$.
  \end{thm}

As a consequence of Theorem \ref{thm0} and Theorem \ref{thm2}, one gets that $|\nb \ln w|^{\frac{p-\gz+2}2}\in W^{1,2}_\loc$ for $\gz<3+\frac{p-1}{n-1}$ or $\gz<3+\frac{p-1}{N-1}$, {\color{black}while in \cite{wz10,dd16}, one has $|\nb \ln w|^{\frac{p-\gz+2}2}\in W^{1,2}_\loc$ for all $\gz <2$ (see \eqref{secor} and the line above \eqref{mest0}). Thus our range for $\gz$ obviously  improves the one obtained in \cite{wz10,dd16} respectively.}

Now we sketch the ideas to prove Theorem \ref{thm0} and Theorem \ref{thm2}.
Note that when $N=n$ and $h\equiv 1$, we have $Ric_h^N=Ric_g$, and hence
 Theorem \ref{thm0} corresponds to the special case $N=n$ and $h\equiv1$ in  Theorem \ref{thm2}.
We only need to prove Theorem \ref{thm2}.
%
As usual, we approximate $u=-(p-1)\ln w$ by  smooth  solution  $u^\ez$ to
  the standard approximation/regularized equation \eqref{plap1}, that is,
$$ e^h\ddiv(e^{-h}[|\nb u^\ez|^2+\ez]^{\frac{p-2}2}\nb u^\ez)= [|\nb u^\ez|^2+\ez]^{\frac{p-2}2}|\nb u^\ez|^2 . $$
\begin{enumerate}
\item[(i)]
Using Bochner formula and the approximation equation \eqref{plap1}, for $0<\eta<1/2$ we bound  the {\color{black}integral} of
\begin{equation}\label{xxx}(1-\eta)|\nb^2u^\ez|^2 +    (p-\gz)    \frac{|\nb^2u^\ez \nb u^\ez|^2}{|\nb u^\ez|^2+\ez}
 + (p-2) (2-\gz)    \frac{(\bdz_\fz u^\ez)^2} {[|\nb u^\ez|^2+\ez]^2}
 \end{equation}
from above by
the {\color{black}integral} of $$Ric_g(\nb u^\ez,\nb u^\ez)+\la\nb^2 h\nb u^\ez,\nb u^\ez\ra$$  and other first order {\color{black}terms}, where all {\color{black}integrals} are taken against    $[|\nb u^\ez|^2+\ez]^{  \frac{p-\gz}2  }\phi^2 \,e^{-h}d{\rm vol}_g$ {\color{black} where $\phi \in C_c^\fz(U)$ is a test function and $U \Subset \Omega$}; see Lemma \ref{divp-gz0}.
{\color{black}Here in \eqref{xxx} and in what follows, for any $C^2$ function $f$, $\bdz_\fz f := \la\nabla^2 f \nabla f, \nabla f\ra$.}

\item[(ii)] If $\gz<3+\frac{p-1}{N-1}$, via a fundamental inequality  given in Lemma \ref{keylem3} and the approximation equation \eqref{plap1},  for sufficiently small $\eta>0$ we bound \eqref{xxx} {\color{black} as below}
 $$  \eqref{xxx}\ge
\eta |\nb^2u^\ez|^2-\frac{\la \nb h,\nb u^\ez\ra^2}{N-n}-C\frac1\eta|\nb u^\ez|^4 \quad\mbox{ everywhere};$$
see Lemma \ref{keylem1}. This is crucial to get Theorem \ref{thm2}. Note that  the approach in \cite{wz10,dd16} could not give Lemma \ref{keylem1}; see Remark \ref{finalrmk} for details.

\item[(iii)] Combining (i)\&(ii) together,  the {\color{black}integral} of $ \eta |\nb^2 u^\ez|^2  $  is {\color{black}bounded} from above by
the {\color{black}integral} of $-Ric_h^N(\nb u^\ez,\nb u^\ez)  $ and other first order {\color{black}terms},
where  all  {\color{black}integrals} are taken against   $[|\nb u^\ez|^2+\ez]^{  \frac{p-\gz}2  }\phi^2 \,e^{-h}d{\rm vol}_g$; see Corollary \ref{cor1}.

Under the assumption $ Ric_h^N\ge -\kz$,  in Lemma \ref{fi} we obtain an upper $L^2_\loc$ bound for
$\nb [|\nb u^\ez|^{\frac{p-\gz}2}\nb u^\ez]\phi$ by the {\color{black}integral} of some first order terms, where
 all {\color{black}integrals} are against $\,e^{-h}d{\rm vol}_g$.  
A standard argument then leads to the proof of  Theorem \ref{thm2}.
\end{enumerate}

Finally, we also notice that the Cheng-Yau  gradient estimate \eqref{cyg}
was generalized to positive harmonic functions $w$ in Alexandrov spaces with curvature bounded from below by Zhang-Zhu in \cite{zz12},
where the authors showed $|\nb \ln w|^2 \in W^{1,2}_\loc(\Omega)$  as a key step. {\color{black}Furthermore, one could study the regularity of $p$-harmonic functions in more general metric measure spaces. In these spaces, a natural generalization of the (weighted) Ricci curvature bound is the curvature-dimension condition $RCD(\kz, N)$ in the sense of Bakry-\'{E}mery or Ambrosio-Gigli-Savar\'{e}. The two senses turned out to be equivalent by the work of Erbar-Kuwada-Sturm \cite{eks15} (in the finite dimensional case) and Ambrosio-Gigli-Savar\'{e} \cite{ags15} and the spaces satisfying one of the two equivalent conditions are known as $RCD(\kz, N)$ spaces.}
Some progress was made in
$RCD(\kz, N)$ spaces. The
Cheng-Yau  gradient estimate was established by Jiang in \cite{j} for positive harmonic functions $w$ in $RCD(\kz, N)$ spaces;
recently, Gigli-Violo in \cite{gv21} established $|\nb \ln w|^{\bz/2} \in W^{1,2}_\loc(\Omega)$  under $RCD(0,N)$ spaces if $\bz>\frac{N-2}{N-1}$.
However, when $ p\ne 2$, it remains  open to prove the Cheng-Yau type gradient estimates for positive $p$-harmonic functions in  Alexandrov  spaces and
also $RCD (\kz,N)$ spaces.

\section{Preliminaries}


Let $n\ge2$ and  $M^n$ be a  Riemannian manifold, and $g$ be the Riemannian metric.
By abuse of notation we also write $|\xi|^2=g(\xi,\xi)$ and $\langle\xi,\eta\rangle=g(\xi,\eta)$ for all $ \xi,\eta\in T_xM^n$.
The corresponding Riemannian volume measure is written as $d{\rm vol}_g$,
and the volume of a set $E$ is written as ${\rm vol}_g(E)$.
Denote by  $Ric_g$ the Ricci curvature 2-tensor and write $Ric_g\ge -\kz$
if $Ric_g(\xi,\xi)\ge -\kz |\xi|^2$ for all $ \xi\in T_xM^n$.

For $1<p<\fz$, the $p$-Laplace operator $\Delta_p$ in $M^n$ is given by
 $$\Delta_p f={\rm div}(|\nb f|^{p-2}\nb f )\ \mbox{  $\forall f\in C^2(M^n)$}. $$
Obviously, $\Delta_2$ is exactly the Laplace-Beltrami operator $\Delta$ in $(M^n,g)$.
A function $w$ defined in a domain $\Omega\subset M^n$ is called $p$-harmonic if
$w\in W^{1,p}_\loc(\Omega)$ is a weak solution to the $p$-Laplace equation $\Delta_pw=0$ in $\Omega$, that is,
$$\int_{\Omega}|\nb w|^{p-2}\langle \nabla w,\nabla \phi\rangle d{\rm vol}_g=0\quad\forall \phi\in C^\fz_c(\Omega).$$
Note that
 $2$-harmonic functions are the well-known harmonic functions.

Next we recall some basic facts of    weighted Riemannian manifolds $(M^n,g,e^{-h}d{\rm vol}_g)$, where  the weight $h $ is a positive smooth  function in $M^n$.
The weighted measure $d{\rm vol}_h=e^{-h} d{\rm vol}_g$ can be viewed as the volume form of a suitable conformal change of the metric $g$.
Denote by  ${\rm vol}_h(E)$  the weighted volume of a set $E$.
For $n\le N<\fz$, the corresponding $N$-Bakry-\'{E}mery curvature tensor is
$$  Ric^N_h= Ric_g + \nb^2 h- \frac{\nb h \otimes \nb h}{N-n},$$
where when $N=n$, by convention, $h$ is {\color{black}a} constant function and hence  $Ric^N_h=Ric_g$. 
We say that $(M^n,g,e^{-h}d{\rm vol}_g)$ satisfies the Bakry-\'{E}mery  curvature-dimension condition
$Ric_h^N \ge-\kz$ if  $$Ric_h^N(\xi,\xi)=Ric_g(\xi,\xi) + \langle\nb^2 h \xi,\xi \rangle- \frac{\langle \nb h, \xi\rangle^2}{N-n} \ge-\kz \langle\xi,\xi\rangle\ \forall \xi\in T_xM^n$$
By \cite{q97}, under $Ric^N_h \ge -\kz$, one has the following volume comparison result
\begin{equation}\label{vc}
{\rm vol}_h(B_{2r}(x)) \le C(N)e^{\sqrt{\kz}r} {\rm vol}_h(B_{r}(x))   \quad  \forall x \in M, \ r>0.
\end{equation}


For $1<p<\fz$, the weighted $p$-Laplacian $\bdz_{h,p}$ is defined as
\begin{equation*}
  \bdz_{p,h}f = e^{h} \ddiv (e^{-h} |\nb f|^{ p-2} \nb f) = \bdz_p f - |\nb f|^{p-2}\la \nb f , \nb h\ra \quad \forall f\in C^2(M^n).
\end{equation*}
In the case  $p=2$, one   writes $ \Delta_{2,h }$ as $\Delta_h$,
 and hence $$\Delta_h f=\Delta f-\langle\nb h,\nb f\rangle. $$
A   function $w$ in a domain $\Omega\subset M^n$  is called as a weighted $p$-harmonic function if
  $w \in W^{1,p}_{\loc}(\Omega)$  is a weak solution to the weighted $p$-harmonic equation $\Delta_{p,h}w=0$ in $\Omega$, that is,
  \begin{equation}  \label{defp}
  \int_\Omega  |\nb w|^{p-2} \la \nb w, \nb \phi \ra \,e^{-h}d{\rm vol}_g=0\quad\forall \phi\in C^\fz_c(\Omega).
  \end{equation}
By a density argument, we can relax $\phi\in C^\fz_c(\Omega)$ to $\phi\in W^{1,p}_0(\Omega)$ in \eqref{defp}.

We also recall the following Bochner formula in $(M^n,g,e^{-h}d{\rm vol}_g)$:
\begin{equation}\label{Boch}
    \frac12 \bdz_h |\nb f|^2 = |\nb^2 f|^2+  \la \nb f, \nb \bdz_h f \ra +
 Ric _g (\nb f, \nb f)+  \langle \nb^2 h\nabla f,\nb f\rangle \quad\forall f\in C^3(M),
  \end{equation}
which will be used in Section 3.

 Finally, we recall the following fundamental inequality; see for example \cite{wz10,dd16,ss20}. 
For {\color{black}the reader's convenience we include it} here.
  Recall   that    $\Delta_\fz f=\langle \nabla^2 f  \nabla f,\nabla f\rangle .$

\begin{lem}\label{keylem3}
Let $n\ge2$ and $\Omega$ be a domain of $M^n$. For any $f\in C^2(\Omega)$, we have
 \begin{align}\label{keyin1}
 |\nb f|^4 |\nb^2 f|^2 \ge 2  |\nb f|^2 | \nb^2 f \nb f|^2+ \frac{[|\nb f|^2 \bdz f-\bdz_\fz f]^2}{n-1}-(\bdz_\fz f)^2 \ \mbox{in}\ \Omega,
 \end{align}
where when $n=2$, ``$\ge$" becomes  ``$=$".
\end{lem}

\begin{proof}
  It suffices to prove that for any {\color{black}symmetric} $n\times n$ matrix $A$ one has
\begin{align}\label{keyin3}|A|^2|\xi|^4\ge \frac1{n-1}({\rm tr} A |\xi|^2 -\langle A\xi ,\xi\rangle)^2+2|A\xi  |^2|\xi|^2-\langle A\xi  ,\xi\rangle^2 \quad  \forall \xi\in\rn.\end{align}
Note that if $\xi=0$, \eqref{keyin3} holds obviously.
Below assume that   $\xi\ne 0$.
Up to a scaling  we may assume  $|\xi|=1$. By  a change of coordinates,
 we may further assume $\xi=e_n=(0,\cdots,0,1)$; in this case, \eqref{keyin3} reads as
\begin{align*}
|A|^2 \ge \frac1{n-1}({\rm tr} A  -\langle A e_n,e_n\rangle)^2+2|Ae_n  |^2 -\langle Ae_n  ,e_n\rangle^2.
\end{align*}
Denoting by  $A_{n-1}$   the $(n-1)$ order principal submatrix of $A$, one has
 $$|A|^2=|A_{n-1}|^2+2|Ae _n|^2-\langle Ae _n,e_n\rangle^2.$$
Noting that $$   |A_{n-1}|^2\ge\frac1{n-1}({\rm tr} A_{n-1})^2  = \frac1{n-1}({\rm tr} A -\langle Ae_n  ,e_n\rangle)^2,$$
where when $n=2$, one has $    |A_{n-1}|^2= ({\rm tr} A_{n-1})^2$,  one concludes  \eqref{keyin1}.
\end{proof}

%
%
%
%
%

\section{Proof of  Theorem \ref{thm2}}

Let $w$ be a positive weighted $p$-harmonic function in a domain $\Omega$.
Set $u= -(p-1)\ln w$. Then $u$ is a weak solution to  the equation
\begin{equation}\label{plap3}
  \bdz_{p} u- |\nb u|^{p-2}\la \nb u , \nb h\ra =  |\nb u|^{p}\quad  \mbox{ in $\Omega$},
\end{equation}
that is,
$$-
\int_{\Omega} |\nb u|^{p-2}\langle \nb u,\nb\phi\rangle e^{-h}d{\rm vol}_g =
\int_{\Omega} |\nb u|^{p }\phi  e^{-h}d{\rm vol}_g \quad \forall \phi\in C^\fz_c(\Omega).$$

Given any smooth domain $U\Subset \Omega$  and $\ez\in (0,1]$,    consider the approximation/regularized equation defined by
\begin{equation}\label{approx}
  e^h\ddiv(e^{-h}[|\nb v|^2+\ez]^{\frac{p-2}2}\nb v)= [|\nb v|^2+\ez]^{\frac{p-2}2} |\nb v|^2 \quad \mbox{ in $U$; $v=u$  on $\partial U$.}
\end{equation}

It is well known that if $u$ is the solution to \eqref{plap3}, then $u\in C^{1,\alpha}(\Omega)$ {\color{black} for some $\az \in(0,1)$}; see \cite{d83,l83,t84,u77}. {\color{black}Moreover, in the following lemma,
we summarize some properties of the solution $u$ to \eqref{plap3} and $u^\ez$ to \eqref{plap1}, which result from \cite{d83} as a special case. See also \cite{t84}.

\begin{lem}\label{d1983}
  For any $\ez \in (0,1]$, there exists a unique solution  $u^\ez\in C^\fz(U )\cap C^0(\overline U )$ to \eqref{plap1}, and moreover, $u^\ez\to u$ in $C^{0}(\overline U)$ and $u^\ez\to u$ in $C^{1,\az}(V)$ uniformly in $\ez>0$
  as $\ez\to0$ for all $V \Subset U$ where $u$ is the solution to \eqref{plap3}.
\end{lem}
To show Lemma \ref{d1983}, we just need to check that equations \eqref{plap3} and \eqref{plap1} are special cases of those considered in \cite{d83}. We put this verification in the appendix.}

{\color{black}By Lemma \ref{d1983}, the solution $u^\ez$ to \eqref{approx} is $C^\fz$, which implies that $u^\ez$ satisfies \eqref{approx} pointwise. Hence by a direct computation, \eqref{approx} is equivalent to}
\begin{equation}\label{plap1}
\bdz_h u^\ez + (p-2)\frac{\bdz_\fz u^\ez}{|\nb u^\ez|^2+\ez} =   |\nb u^\ez|^2 \quad \mbox{ in $U$; $u^\ez=u$  on $\partial U$.} \end{equation}

To prove Theorem \ref{thm2} we first build up the {\color{black}following} upper bound.

\begin{lem} \label{divp-gz0} {\color{black}Let $u^\ez$ be the solution to \eqref{plap1}.}
For any $\gz\in\rr$, $\eta>0$ and $\phi\in C^\fz_c(U)$, we have
 \begin{align}   \label{bypart5}
& \int_U\left\{(1-\eta)|\nb^2u^\ez|^2 +    (p-\gz)    \frac{|\nb^2u^\ez \nb u^\ez|^2}{|\nb u^\ez|^2+\ez}
 + (p-2) (2-\gz)    \frac{(\bdz_\fz u^\ez)^2} {[|\nb u^\ez|^2+\ez]^2} \right\} \nonumber\\
&\quad\quad\quad \times [|\nb u^\ez|^2+\ez]^{  \frac{p-\gz}2  }\phi^2 \,e^{-h}d{\rm vol}_g  \notag     \\
 &\quad\le -\int_U [Ric_g(\nb u^\ez, \nb u^\ez)+\langle \nb^2 h\nabla u^\ez,\nb u^\ez\rangle][|\nb u^\ez|^2+\ez]^{ \frac{p-\gz}2 }\phi^2\,e^{-h}d{\rm vol}_g\notag
     \\
  &  \qquad   +
    C(p,\gz)\frac1\eta\int_{U }      ( [|\nb u^\ez|^2+\ez]^{\frac{p-\gz}2+1 }|\nb \phi|^2 + [|\nb u^\ez|^2+\ez]^{\frac{p-\gz}2 +2}\phi^2)\,e^{-h}d{\rm vol}_g.
\end{align}
\end{lem}

To prove this, we need the following identity.

\begin{lem} {\color{black}For any $v \in C^3(U)$} and $\psi\in C^\fz_c(U)$, one has
\begin{align}\label{intformula}
 \int_U |\nb^2v|^2
\psi\,e^{-h}d{\rm vol}_g
  &  =-\int_U\la \nb^2v \nb v - \bdz_h v \nb v, \nb\psi \ra \,e^{-h}d{\rm vol}_g+\int_U  (\bdz_h v)^2
\psi\,e^{-h}d{\rm vol}_g\nonumber \\
&  \quad-\int_U[  Ric_g(\nb v, \nb v)+\langle \nb^2 h\nabla v,\nb v\rangle]
\psi\,e^{-h}d{\rm vol}_g.
\end{align}

\end{lem}

\begin{proof}
Applying the Bochner formula to $v$, one has
\begin{equation*} 
|\nb^2 v|^2
+Ric _g (\nb v, \nb v)=
    \frac12 \bdz_h |\nb v|^2 -  \la \nb v, \nb \bdz_h v \ra -
   \langle \nb^2 h\nabla v,\nb v\rangle
  \end{equation*}
and hence
\begin{align*}
 |\nb^2 v|^2
&=
    [\frac12 \bdz_h |\nb v|^2 -(\bdz_h v)^2 -  \la \nb v, \nb \bdz_h v \ra] +   (\bdz_h v)^2\\
&\quad-[Ric _g (\nb v, \nb v)+ \langle \nb^2 h\nabla v,\nb v\rangle].
  \end{align*}
By this, to get \eqref{intformula}, it suffices to show the following identity
\begin{align}\label {intbypart}
& \int_U   [\frac12 \bdz_h |\nb v|^2 -(\bdz_h v)^2 -  \la \nb v, \nb \bdz_h v \ra] \psi e^{-h}d{\rm vol}_g\nonumber\\
&\quad=-\int_U\la \nb^2v \nb v - \bdz_h v \nb v, \nb\psi \ra \,e^{-h}d{\rm vol}_g.  \end{align}

Note that
\begin{align*} -[( \bdz_h  v)^2+\langle   \nb  v  ,  \nb (\bdz_h  v) \ra]&= -e^{h}{\rm div}(e^{-h} \nb  v)  (\bdz_h  v)   -e^{h}\langle  e^{-h} \nb  v  ,  \nb (\bdz_h  v) \ra\\
&= -e^{h}{\rm div}(e^{-h} \nb  v  \bdz_h  v).
\end{align*}
Via integration by parts, one has
\begin{align*} -\int_U[( \bdz_h  v)^2+\langle   \nb  v  ,  \nb (\bdz_h  v) \ra]\psi  \,e^{-h}d{\rm vol}_g
&=-\int_U {\rm div}(e^{-h} \nb  v  \bdz_h  v)\psi  \, d{\rm vol}_g\\
&=
 \int_U\la  \Delta_hv\nb v ,\nb\psi\rangle \,e^{-h}d{\rm vol}_g.
\end{align*}
Similarly, via integration by parts one also has
\begin{align*}\frac12\int_U \bdz_h |\nb v|^2\psi e^{-h}\,d{\rm vol}_g&=
\int_U \frac12{\rm div}(e^{-h}\nb |\nb v|^2)\psi  \,d{\rm vol}_g\\
&= -
\int_U   \frac12\langle e^{-h}\nb |\nb v|^2,\nb\psi \rangle\,d{\rm vol}_g\\
&= -\int_U    \langle \nb^2 v   \nb v,\nb\psi \rangle\,e^{-h}d{\rm vol}_g.
\end{align*}
Combining together we obtain \eqref{intbypart} and hence, \eqref{intformula} as desired.
\end{proof}

We are ready prove Lemma \ref{divp-gz0} as below.

\begin{proof}[Proof of Lemma \ref{divp-gz0}]
 Taking {\color{black}$v=u^\ez$ and} $\psi= [|\nb u^\ez|^2+\ez]^{ \frac{p-\gz}2 }\phi^2$ in \eqref{intformula} we get
\begin{align}   \label{bypart0}
 \int_U &|\nb^2u^\ez|^2 [|\nb u^\ez|^2+\ez]^{ \frac{p-\gz}2 }\phi^2\,e^{-h}d{\rm vol}_g \notag\\
  &  =-\int_U\la \nb^2u^\ez \nb u^\ez - \bdz_h u^\ez \nb u^\ez, {\color{black}\nb[[}|\nb u^\ez|^2+\ez]^{ \frac{p-\gz}2 }\phi^2] \ra \,e^{-h}d{\rm vol}_g \notag\\
&\quad + \int_U (\bdz_h u^\ez)^2 [|\nb u^\ez|^2+\ez]^{ \frac{p-\gz}2 }\phi^2\,e^{-h}d{\rm vol}_g \notag\\
 &\quad-\int_U[ Ric(\nb u^\ez, \nb u^\ez)+\langle \nb^2 h\nabla u^\ez,\nb u^\ez\rangle][|\nb u^\ez|^2+\ez]^{ \frac{p-\gz}2 }\phi^2\,e^{-h}d{\rm vol}_g.
\end{align}

To bound the second term in the {\color{black}right-hand} side in \eqref{bypart0}, recalling {\color{black}\eqref{plap1}, that is,}
 \begin{equation}\label{xxxx}\bdz_h u^\ez=  |\nb u^\ez|^2 - (p-2) \frac{\bdz_\fz u^\ez}{|\nb u^\ez|^2+\ez},
\end{equation}
by Cauchy-Schwarz'{\color{black}s} inequality one has
\begin{align*}
(\bdz_h u^\ez)^2\le   (p-2)^2 \frac{(\bdz_\fz u^\ez)^2 }{[|\nb u^\ez|^2+\ez]^2}+
\frac\eta 4  |\nb^2 u^\ez |^2 + C(p)\frac 1\eta |\nb u^\ez|^4 ,\end{align*}
where $0<\eta<1$ is any constant. Thus
\begin{align}\label{secondterm}
\int_U (\bdz_h u^\ez)^2 [|\nb u^\ez|^2+\ez]^{ \frac{p-\gz}2 }\phi^2\,e^{-h}d{\rm vol}_g&\le
 (p-2)^2 \int_U (\bdz_\fz u^\ez)^2   [|\nb u^\ez|^2+\ez]^{ \frac{p-\gz}2-2 }\phi^2\,e^{-h}d{\rm vol}_g\nonumber\\
&\quad+ \frac\eta 4 \int_U  |\nb^2 u^\ez  |^2   [|\nb u^\ez|^2+\ez]^{ \frac{p-\gz}2 }\phi^2\,e^{-h}d{\rm vol}_g\nonumber\\
&\quad+ \frac{C(p)}\eta \int_U    [|\nb u^\ez|^2+\ez]^{ \frac{p-\gz}2 +2 } \phi^2\,e^{-h}d{\rm vol}_g.
\end{align}

The first term in the right-hand side in \eqref{bypart0}  is further written as
\begin{align}\label{bypart}
&-\int_U\la \nb^2u^\ez \nb u^\ez - \bdz_h u^\ez \nb u^\ez, {\color{black}\nb[[}|\nb u^\ez|^2+\ez]^{ \frac{p-\gz}2 }\phi^2] \ra \,e^{-h}d{\rm vol}_g  \notag \\
 &\quad=    -   (p-\gz) \int_{U}  \frac{|\nb^2u^\ez \nb u^\ez|^2}{|\nb u^\ez|^2+\ez} [|\nb u^\ez|^2+\ez]^{ \frac{p-\gz}2 }\phi^2\,e^{-h}d{\rm vol}_g \notag \\
  &\quad\quad  +
   (p-\gz) \int_{U}  \bdz_h u^\ez  \frac{\bdz_\fz u^\ez}{|\nb u^\ez|^2+\ez} [|\nb u^\ez|^2+\ez]^{  \frac{p-\gz}2 }\phi^2\,e^{-h}d{\rm vol}_g  \notag \\
 &\quad\quad  -
  \int_{U} \la \nb^2u^\ez \nb u^\ez  , \nb \phi^2\ra [|\nb u^\ez|^2+\ez]^{\frac{p-\gz}2}\,e^{-h}d{\rm vol}_g\nonumber\\
 &\quad\quad  +
  \int_{U} \la   \bdz_h u^\ez \nb u^\ez, \nb \phi^2\ra [|\nb u^\ez|^2+\ez]^{\frac{p-\gz}2}\,e^{-h}d{\rm vol}_g.
\end{align}
Using  \eqref{xxxx}
and Cauchy-Schwarz's inequality, we obtain
the following upper bound for the second term in \eqref{bypart}:
\begin{align}  \label{bypart1}
  (p  &-\gz)\int_{U} \bdz_h u^\ez  \frac{\bdz_\fz u^\ez}{|\nb u^\ez|^2+\ez} [|\nb u^\ez|^2+\ez]^{  \frac{p-\gz}2 }\phi^2\,e^{-h}d{\rm vol}_g \notag \\
  & = -(p-\gz)(p-2)\int_{U} (\bdz_\fz u^\ez)^2  [|\nb u^\ez|^2+\ez]^{  \frac{p-\gz}2-2 }\phi^2 \,e^{-h}d{\rm vol}_g \notag \\
  & \quad +(p-\gz)
   \int_{U} \bdz_\fz u^\ez |\nb u^\ez|^2 [|\nb u^\ez|^2+\ez]^{  \frac{p-\gz}2-1 }\phi^2\,e^{-h}d{\rm vol}_g  \notag \\
   & \le   -(p-\gz)  (p-2)   \int_{U} (\bdz_\fz u^\ez)^2  [|\nb u^\ez|^2+\ez]^{  \frac{p-\gz}2-2 }\phi^2 \,e^{-h}d{\rm vol}_g  \notag \\
&\quad+ \frac\eta 4\int_{U}  |\nb ^2 u^\ez |^2   [|\nb u^\ez|^2+\ez]^{  \frac{p-\gz}2  }\phi^2 \,e^{-h}d{\rm vol}_g  \notag \\
   &\quad +\frac{C(p) }{\eta}|p-\gz|^2\int_{U }  [|\nb u^\ez|^2+\ez]^{\frac{p-\gz}2 +2}\phi^2\,e^{-h}d{\rm vol}_g.
\end{align}  For the third term in the right-hand side of \eqref{bypart},
  by {\color{black}Cauchy}-Schwarz's inequality,  one has
  \begin{align} \label{xx2}
& \lf |\int_{U}  \la \nb^2u^\ez \nb u^\ez, \nb \phi^2\ra[|\nb u^\ez|^2+\ez]^{\frac{p-\gz}2}\,e^{-h}d{\rm vol}_g \lr | \notag\\
& \quad\le \frac  \eta  4\int_{U }  |\nb^2 u^\ez|^2
  [|\nb u^\ez|^2+\ez]^{\frac{p-\gz}2 }      \phi^2  \,e^{-h}d{\rm vol}_g  + C\frac1\eta\int_{U }       [|\nb u^\ez|^2+\ez]^{\frac{p-\gz}2+1 }|\nb \phi|^2\,e^{-h}d{\rm vol}_g.
    \end{align}
For  the fourth term in the right-hand side of \eqref{bypart},
 in a similar way, using  \eqref{xxxx},  one has
  \begin{align} \label{xx3}
&\lf |\int_{U} \la  \bdz_h u^\ez \nb u^\ez, \nb \phi^2\ra[|\nb u^\ez|^2+\ez]^{\frac{p-\gz}2}\,e^{-h}d{\rm vol}_g \lr |\notag\\
&\quad= \lf |\int_{U} \left\la    |\nb u^\ez|^2 \nb u^\ez - (p-2) \frac{\bdz_\fz u^\ez}{|\nb u^\ez|^2+\ez} \nb u^\ez, \nb \phi^2\right\ra[|\nb u^\ez|^2+\ez]^{\frac{p-\gz}2}\,e^{-h}d{\rm vol}_g \lr |\notag\\
& \quad\le \frac  \eta  4\int_{U }   |\nb^2u^\ez  |^2
  [|\nb u^\ez|^2+\ez]^{\frac{p-\gz}2 }      \phi^2  \,e^{-h}d{\rm vol}_g \notag \\
      &\quad\quad+ C(p) \frac1\eta\int_{U }      ( [|\nb u^\ez|^2+\ez]^{\frac{p-\gz}2+1 }|\nb \phi|^2 + [|\nb u^\ez|^2+\ez]^{\frac{p-\gz}2 +2}\phi^2)\,e^{-h}d{\rm vol}_g.
    \end{align}
From \eqref{xx3}, \eqref{xx2},  \eqref{bypart1} and \eqref{bypart}  we attain
 \begin{align}\label{bypart3}
&-\int_U\la \nb^2u^\ez \nb u^\ez - \bdz_h u^\ez \nb u^\ez, {\color{black}\nb[[}|\nb u^\ez|^2+\ez]^{ \frac{p-\gz}2 }\phi^2] \ra \,e^{-h}d{\rm vol}_g  \notag \\
 &\quad=   \frac 34  \eta\int_{U}   |\nb^2u^\ez |^2  [|\nb u^\ez|^2+\ez]^{ \frac{p-\gz}2 }\phi^2\,e^{-h}d{\rm vol}_g \notag \\
 &\quad\quad-   (p-\gz)  \int_{U}  \frac{|\nb^2u^\ez \nb u^\ez|^2}{|\nb u^\ez|^2+\ez} [|\nb u^\ez|^2+\ez]^{ \frac{p-\gz}2 }\phi^2\,e^{-h}d{\rm vol}_g \notag \\
  &\quad\quad   -(p-\gz)(p-2) \int_{U} (\bdz_\fz u^\ez)^2  [|\nb u^\ez|^2+\ez]^{  \frac{p-\gz}2-2 }\phi^2 \,e^{-h}d{\rm vol}_g  \notag
     \\
 &\quad\quad
   + \frac{C(p)}\eta\int_{U }      ( [|\nb u^\ez|^2+\ez]^{\frac{p-\gz}2+1 }|\nb \phi|^2 + [|\nb u^\ez|^2+\ez]^{\frac{p-\gz}2 +2}\phi^2)\,e^{-h}d{\rm vol}_g.
\end{align}

Obviously from \eqref{bypart3}, \eqref{secondterm} and  \eqref{bypart0} we conclude \eqref{bypart5}.
\end{proof}

If $\gz<3+\frac{p-1}{N-1}$, we get the following   pointwise lower bound.  Recall that when $N=n$, we always assume that $h$ is a constant function and $\frac{\la \nb u^\ez, \nb h\ra^2}{N-n}=0$.

\begin{lem}\label{keylem1}{\color{black}Let $u^\ez$ be the solution to \eqref{plap1}.}
If $\gz<3+\frac{p-1}{N-1}$ for some $N\ge n$, then for sufficiently small $\eta>0$ we have
\begin{align}\label{lbdd2}
&(1-\eta)|\nb^2u^\ez|^2 +   (p-\gz)    \frac{|\nb^2u^\ez \nb u^\ez|^2}{|\nb u^\ez|^2+\ez}
 + (p-2) (2-\gz)    \frac{(\bdz_\fz u^\ez)^2} {[|\nb u^\ez|^2+\ez]^2}\nonumber\\
&\quad\ge  \eta |\nb^2u^\ez|^2
  -   \frac{\la \nb u^\ez, \nb h\ra^2}{N-n}  - C(n,N,p,\gz)\frac1\eta |\nb u^\ez|^4.
    \end{align}
\end{lem}
To prove this,  we need the following pointwise lower bound for  $|\nb^2 u^\ez|^2 |\nb u^\ez|^4$.

\begin{lem}\label{keylem2}{\color{black}Let $u^\ez$ be the solution to \eqref{plap1}.}
If  $N\ge n$, then for $0<\eta<1$ we have
\begin{align}\label{lbdd-x1}
     (1+\eta) |\nb^2 u^\ez|^2 |\nb u^\ez|^4& \ge
2  | \nb^2 u^\ez\nb u^\ez  |^2  |\nb u^\ez|^2 + \left(\frac 1{N-1}\left[ (p-2)\frac{|\nb u^\ez|^2}{|\nb u^\ez|^2+\ez}+1\right]^2-1\right)  (\bdz_\fz u^\ez)^2\notag \\
     & \quad   - (1+\eta) \frac{\la \nb u^\ez, \nb h\ra^2}{N-n}  |\nb u^\ez|^4
 - C(n,N,p)\frac1\eta |\nb u^\ez|^8.
    \end{align}
\end{lem}

\begin{proof}
Applying \eqref{keyin1} to $u^\ez$ one has
\begin{align}\label{key3-ez}
|\nb^2 u^\ez|^2 |\nb u^\ez|^4  \ge  2 |\nb u^\ez|^2 | {\color{black}\nb^2} u^\ez \nb u^\ez    |^2+
\frac{[|\nb u^\ez|^2 \bdz u^\ez-\bdz_\fz u^\ez]^2}{n-1}-(\bdz_\fz u^\ez)^2
 \end{align}
By \eqref{xxxx} and $\Delta u^\ez=\Delta_h u^\ez+\langle\nb h,\nb u^\ez\rangle$, we have
  $$\bdz  u^\ez= |\nb  u^\ez|^2+\langle \nb {\color{black}u^\ez},\nb h\rangle-(p-2)\frac{ \bdz_\fz u^\ez}{|\nb u^\ez|^2+\ez}.$$
Thus
\begin{align*}
      |\nb u^\ez|^2 \bdz u^\ez-\bdz_\fz u^\ez
&= |\nb u^\ez|^2 \lf ( |\nb  u^\ez|^2+\la \nb u^\ez, \nb h \ra \lr)-\left[ (p-2)\frac{|\nb u^\ez|^2}{|\nb u^\ez|^2+\ez}+1\right]\bdz_\fz u^\ez,
    \end{align*}
and hence,
    \begin{align}\label{ta4}
      [|\nb u^\ez|^2 \bdz u^\ez-\bdz_\fz u^\ez]^2
       & =  \left[ (p-2)\frac{|\nb u^\ez|^2}{|\nb u^\ez|^2+\ez}+1\right]^2(\bdz_\fz u^\ez)^2\nonumber\\
&\quad +|\nb u^\ez|^4\lf ( |\nb  u^\ez|^2+\la \nb u^\ez, \nb h \ra \lr)^2 \notag \\
       & \quad -2\left[ (p-2)\frac{|\nb u^\ez|^2}{|\nb u^\ez|^2+\ez}+1\right]  |\nb  u^\ez|^4
 \bdz_\fz u^\ez \nonumber\\
&\quad  -
  2\left[ (p-2)\frac{|\nb u^\ez|^2}{|\nb u^\ez|^2+\ez}+1\right] |\nb u^\ez|^2\bdz_\fz u^\ez  \la \nb u^\ez, \nb h \ra\nonumber\\
&=:I_1+I_2+I_3+I_4.
    \end{align}

Note that
\begin{equation}\label{xxxx-2}\left|(p-2)\frac{|\nb u^\ez|^2}{|\nb u^\ez|^2+\ez}+1\right|^2\le 4p^2,
\end{equation}
which can be obtained by considering $p>2$ and $1<p<2$ separately.
Using this,  Cauchy-Schwarz inequality, for   $0< \eta<1$, we have
    \begin{align} \label{low3}
   I_3
    &
\ge - \eta |\nb^2 u^\ez  |^2 |\nb u^\ez|^4  -C(p)\frac1{\eta} |\nb u^\ez|^8.
    \end{align}

If $h $ is a constant function and hence $\nb h=0$, $I_2\ge0$ and $I_4=0$,  dividing by $n-1$ in both sides of \eqref{ta4}, by \eqref{low3} one has
    \begin{align*}
     \frac{ [|\nb u^\ez|^2 \bdz u^\ez-\bdz_\fz u^\ez]^2}{n-1}
        \ge -\eta|\nb^2 u^\ez|^2|\nabla u^\ez|^4+ \left[ (p-2)\frac{|\nb u^\ez|^2}{|\nb u^\ez|^2+\ez}+1\right]^2 {\color{black}\frac{(\bdz_\fz u^\ez)^2}{n-1}}  -\frac{C(p)}\eta|\nabla u^\ez|^8.
    \end{align*}
Plugging this in  \eqref{key3-ez}, noting $N=n$,    and  adding $ \eta |\nb^2 u^\ez|^2 |\nb u^\ez|^4 $ in both side, one concludes \eqref{lbdd-x1}.

If $h$ is not a constant function,
set  $\eta_1= \frac{N -n}{N -1}$. Then
    \begin{align} \label{eta1}1- \eta_1  =\frac{n-1}{N -1}>0\quad \mbox{and}\quad  1-\frac{1}{\eta_1}= -\frac{n-1}{N -n}<0. \end{align}
For any $0<\eta<1$ one has
    \begin{align}\label{ta3}
      I_2 &\ge |\nb u^\ez|^4 \la \nb u^\ez, \nb h \ra^2+2|\nb u^\ez|^6 \la \nb u^\ez, \nb h \ra \nonumber\\
&\ge  [1+\eta (1-  \frac1{\eta_1})] |\nb u^\ez|^4
\la \nb u^\ez, \nb h \ra^2  -\frac1{\eta |1-  \frac1{\eta_1}| } |\nb u^\ez|^8.
    \end{align}
Using Cauchy-Schwarz inequality,  we have
    \begin{align}\label{xxxx-1}
   I_4
    &\ge -  \eta_1 \left[ (p-2)\frac{|\nb u^\ez|^2}{|\nb u^\ez|^2+\ez}+1\right]^2(\bdz_\fz  u^\ez)^2    -    \frac{1}{\eta_1}\la \nb u^\ez, \nb h \ra^2 |\nb u^\ez|^4
    \end{align}
Dividing by $n-1$ in both sides of \eqref{ta4}, by \eqref{low3}, \eqref{ta3} and \eqref{xxxx-1}
 one has
 \begin{align*}
     \frac{ [|\nb u^\ez|^2 \bdz u^\ez-\bdz_\fz u^\ez]^2}{n-1}
 & \ge  - \eta |\nb^2 u^\ez  |^2 |\nb u^\ez|^4 +\frac{1-\eta_1}{n-1}\left[ (p-2)\frac{|\nb u^\ez|^2}{|\nb u^\ez|^2+\ez}+1\right]^2(\bdz_\fz u^\ez)^2\nonumber\\
&\quad +(1 +\eta)\frac{1-\frac1{\eta_1}}{n-1}
\la \nb u^\ez, \nb h \ra^2|\nb u^\ez|^4 - C(n,N,p)\frac1{\eta }  |\nb u^\ez|^8
    \end{align*}
By \eqref{eta1},
 \begin{align*}
  \frac{ [|\nb u^\ez|^2 \bdz u^\ez-\bdz_\fz u^\ez]^2}{n-1} &\ge  - \eta |\nb^2 u^\ez  |^2 |\nb u^\ez|^4 +\frac1{N -1} \left[ (p-2)\frac{|\nb u^\ez|^2}{|\nb u^\ez|^2+\ez}+1\right]^2(\bdz_\fz u^\ez)^2\nonumber\\
&\quad -(1 +\eta)\frac1{N -n}
\la \nb u^\ez, \nb h \ra^2|\nb u^\ez|^4 - C(n,N,p)\frac1{\eta }  |\nb u^\ez|^8.
    \end{align*}
 Plugging this in
  \eqref{key3-ez}, and adding $ \eta |\nb^2 u^\ez|^2 |\nb u^\ez|^4 $ in both side,  we conclude \eqref{lbdd-x1} as desired.
\end{proof}

We now prove Lemma \ref{keylem1} by using Lemma \ref{keylem2}.

\begin{proof}[Proof of Lemma \ref{keylem1}] Given any point $x\in U$, if $\nb u^\ez(x)=0$, then \eqref{lbdd2} holds trivially.
Below we assume that $\nb u^\ez(x)\ne 0$. At such point $x$, we already have \eqref{lbdd-x1} in Lemma \ref{keylem2}.
Dividing  by  $ |\nb u^\ez|^4 $ in both sides of \eqref{lbdd-x1},  for $0<\eta<1/2$ we obtain
    \begin{align*}
    (1+\eta) |\nb^2 u^\ez|^2   & \ge   2   \frac{| \nb^2 u^\ez\nb u^\ez  |^2}{  |\nb u^\ez|^2}   +
\left(\frac 1{N-1}\left[ (p-2)\frac{|\nb u^\ez|^2}{|\nb u^\ez|^2+\ez}+1\right]^2-1\right)  \frac{(\bdz_\fz u^\ez)^2
 }{|\nb u^\ez|^4} \nonumber\\
&\quad -(1+\eta)  \frac{\la \nb u^\ez, \nb h\ra^2}{N -n}  -  \frac{C(n,N,p)}\eta |\nb u^\ez|^4  .
    \end{align*}
In both sides, multiplying by $\frac{1-2\eta}{1+\eta} >0 $   and  adding
$$\eta|\nb ^2 u^\ez|^2+ (p-\gz)    \frac{|\nb^2u^\ez \nb u^\ez|^2}{|\nb u^\ez|^2+\ez}
 + (p-2) (2-\gz)    \frac{(\bdz_\fz u^\ez)^2} {[|\nb u^\ez|^2+\ez]^2},$$
we get
\begin{align}\label{lbdd3}
&(1-\eta)|\nb^2u^\ez|^2 +    (p-\gz)    \frac{|\nb^2u^\ez \nb u^\ez|^2}{|\nb u^\ez|^2+\ez}
 + (p-2) (2-\gz)    \frac{(\bdz_\fz u^\ez)^2} {[|\nb u^\ez|^2+\ez]^2}\nonumber\\
&\ge \eta |\nb^2u^\ez|^2+ \left\{\frac{1-2\eta}{1+\eta} 2  +
( p-\gz )   \frac{|\nb u^\ez|^2}{|\nb u^\ez|^2+\ez}  \right\}\frac{| \nb^2 u^\ez\nb u^\ez  |^2}{  |\nb u^\ez|^2}\nonumber\\
&\quad+  \left\{\frac{1-2\eta}{1+\eta} \left(\frac 1{N-1}\left[ (p-2)\frac{|\nb u^\ez|^2}{|\nb u^\ez|^2+\ez}+1\right]^2-1\right)
+ (p-2) (2-\gz)    \frac{|\nb u^\ez|^4} {[|\nb u^\ez|^2+\ez]^2}\right\}  \frac{(\bdz_\fz u^\ez)^2
 }{|\nb u^\ez|^4}\nonumber \\
  & \quad - (1-2\eta)   \frac{\la \nb u^\ez, \nb h\ra^2}{N -n}  - C(n,N,p)\frac1\eta |\nb u^\ez|^4 \nonumber\\
&=:I_1+I_2+I_3+I_4+I_5.
    \end{align}

Recall that if $N=n$ that is, $h$ is a constant function, $I_4=0$ by our convention.
If $N>n$  that is, $h$ is not a constant,  then by  $ 1-2\eta<1 $, we   have
\begin{equation}\label{Nn}I_4\ge -\frac{\la \nb u^\ez, \nb h\ra^2}{N-n}.
\end{equation}

To bound $I_2+I_3$ from below, since $ \gz < 3+\frac{p-1}{N-1}$ and $N\ge 2$ implies
$$p+2-\gz>p+2-  {\color{black}3-\frac{p-1}{N-1}}=(p-1)(1-\frac1{N-1})\ge0,$$
we can find $0<\hat \eta(p,\gz)<1/2$
such that  for $0<\eta<\hat \eta$, one has
$p+2\frac{1-2\eta}{1+\eta}-\gz >0 $. Thus the coefficient of $I_2$  satisfies
\begin{align*}
& (p-\gz )\frac{|\nb u^\ez|^2}{|\nb u^\ez|^2+\ez}+2\frac{1-2\eta}{1+\eta} \ge
(p+2\frac{1-2\eta}{1+\eta}-\gz )\frac{|\nb u^\ez|^2}{|\nb u^\ez|^2+\ez} +\frac{1-2\eta}{1+\eta}\frac{\ez}{|\nb u^\ez|^2+\ez}>0.
\end{align*}
Using this and observing  $$ \frac{|\nb^2 u^\ez \nb u^\ez|^2}{|\nb u^\ez|^2} \ge \frac{|\bdz_\fz u^\ez|^2}{|\nb u^\ez|^4},$$   one has
\begin{align*}
&I_2 +I_3 \\
&\ge
 \Bigg\{
(p -\gz )\frac{|\nb u^\ez|^2}{|\nb u^\ez|^2+\ez} +2\frac{1-2\eta}{1+\eta}  \\
&\quad\quad \left.+\frac{1-2\eta}{1+\eta}\left(\frac 1{N-1}\left[ (p-2)\frac{|\nb u^\ez|^2}{|\nb u^\ez|^2+\ez}+1\right]^2-1\right)+ (p-2) (2-\gz)    \frac{|\nb u^\ez|^4} {[|\nb u^\ez|^2+\ez]^2}\right\}  \frac{(\bdz_\fz u^\ez)^2
 }{|\nb u^\ez|^4}\\
&=:H(\eta) \frac{(\bdz_\fz u^\ez)^2
 }{|\nb u^\ez|^4}
 \end{align*}

We claim that there exists $0<  \bar \eta(n,N,p,\gz)<\hat\eta$
such that   $H(\eta)>0$  for all $0<\eta<\bar \eta$.
Assuming this claim holds for the moment, for any $0<\eta<\bar \eta $, one has
$I_2+I_3> 0$.
From this, \eqref{lbdd3} and \eqref{Nn} we conclude  \eqref{lbdd2} as desired.

Finally we prove {\color{black}the} above claim.   It suffices to show that
\begin{align}\label{h0}H(0)&:=
 (p -\gz )\frac{|\nb u^\ez|^2}{[|\nb u^\ez|^2+\ez]} + 2  + \left(\frac 1{N-1}\left[ (p-2)\frac{|\nb u^\ez|^2}{|\nb u^\ez|^2+\ez}+1\right]^2-1\right)\nonumber\\
&\quad\quad\quad+ (p-2) (2-\gz)    \frac{|\nb u^\ez|^4} {[|\nb u^\ez|^2+\ez]^2}\nonumber\\
& >\dz(N,p,\gz),
 \end{align}
where $\dz(N,p,\gz)>0$ is a constant.
Indeed, by \eqref{xxxx-2},  one has
$$  H(\eta)\ge H(0)- 2[1-  \frac{1-2\eta}{1+\eta}] - [1-\frac{1-2\eta}{1+\eta}] [ \frac {4p^2}{N-1}-1]\ge
 \dz(N,p,\gz)-15p^2 \eta.   $$
If $0<\eta<  \bar \eta=:\min\{\hat\eta,\dz(N,p,\gz)/15p^2\}$,  one has $H(\eta)>0$ and hence the claim holds as desired.

We prove \eqref{h0} as below.
Since
$$
\left[ (p-2)\frac{|\nb u^\ez|^2}{|\nb u^\ez|^2+\ez}+1\right]^2= (p-2)^2\frac{|\nb u^\ez|^4}{[|\nb u^\ez|^2+\ez]^2}+2(p-2)\frac{|\nb u^\ez|^2}{[|\nb u^\ez|^2+\ez]} +1,$$
we rewrite
\begin{align*}
 H(0)
    &  = (p-2)[2-\gz+\frac{p-2}{N-1}] \frac{|\nb u^\ez|^4}{[|\nb u^\ez|^2+\ez]^2} + [p-\gz+\frac{2(p-2)}{N-1}] \frac{ |\nb u^\ez|^2}{[|\nb u^\ez|^2+\ez] } +  \frac{N}{N-1}.
\end{align*}
Observing
 $$\frac{ |\nb u^\ez|^2}{[|\nb u^\ez|^2+\ez] }=   \frac{ |\nb u^\ez|^4}{[|\nb u^\ez|^2+\ez]^2 }+
 \frac{ \ez|\nb u^\ez|^2}{[|\nb u^\ez|^2+\ez]^2 }$$
and $$1= \frac{ |\nb u^\ez|^4}{[|\nb u^\ez|^2+\ez]^2 }+2
 \frac{ \ez|\nb u^\ez|^2}{[|\nb u^\ez|^2+\ez]^2 }+
\frac{ \ez^2}{[|\nb u^\ez|^2+\ez]^2 } ,$$
we further write
\begin{align*}
 H(0)
& = \left\{(p-2)[2-\gz+\frac{p-2}{N-1}]+[p-\gz+\frac{2(p-2)}{N-1}] +\frac{N}{N-1}\right\}  \frac{|\nb u^\ez|^4}{[|\nb u^\ez|^2+\ez]^2}\\
&\quad + \left\{[p-\gz+\frac{2(p-2)}{N-1}]+2\frac{N}{N-1}   \right\} \frac{\ez|\nb u^\ez|^2}{[|\nb u^\ez|^2+\ez]^2}\\
&\quad+  \frac{N}{N-1}
\frac{ \ez^2}{[|\nb u^\ez|^2+\ez]^2 }.
\end{align*}
By
a direct calculation, $\gz<3+\frac{p-1}{N-1}$ implies that
\begin{align*}[p-\gz+\frac{2(p-2)}{N-1}]+2\frac{N}{N-1}&>p+\frac{2(p-2)}{N-1}+2\frac{N}{N-1}-3-\frac{p-1}{N-1}\\
&=p-1 +\frac{2(p-2)+2-(p-1)}{N-1}\\
&=(p-1) \frac{N}{N-1}\\
&>0.\end{align*}
Moreover, $\gz<3+\frac{p-1}{N-1}$ also implies that
\begin{align*}
&(p-2)[2-\gz+\frac{p-2}{N-1}]+[p-\gz+\frac{2(p-2)}{N-1}] +\frac{N}{N-1}\\
&= 3(p-1) +\frac{(p-2)^2 +2(p-2) +1}{N-1} -(p-1)\gz\\
&= 3(p-1) +\frac{(p-1)^2  }{N-1}   -(p-1)\gz\\
&=(p-1)[3+\frac{ p-1   }{N-1} -\gz]\\
&>0.
\end{align*}
Thus
\begin{align*}H(0)&>(p-1)[3+\frac{ p-1   }{N-1} -\gz]\frac{|\nb u^\ez|^4}{[|\nb u^\ez|^2+\ez]^2}+ \frac{N}{N-1}
\frac{ \ez^2}{[|\nb u^\ez|^2+\ez]^2 }\\
&\ge \frac12
\min\left\{(p-1)[3+\frac{ p-1   }{N-1} -\gz], \frac{N}{N-1}\right\}\\
&=:\dz(N,p,\gz)\\
&>0\end{align*}
that is, \eqref{h0} holds.
\end{proof}

Combining \eqref{lbdd2} and \eqref{bypart5} we have the following. Recall that
$$ Ric_h^N (\nb u^\ez, \nb u^\ez)=
 Ric_g(\nb u^\ez, \nb u^\ez)+\la\nb ^2h\nb u^\ez,\nb u^\ez\ra-\frac{\la \nb u^\ez, \nb h \ra^2}{N-n}. $$
\begin{cor} \label{cor1} {\color{black}Let $u^\ez$ be the solution to \eqref{plap1}.}
If $\gz<3+\frac{p-1}{N-1}$ for some $N\ge n$, then for sufficiently small $\eta>0$ one has
\begin{align}\label{e3.x8}
 \eta\int_U &  |\nb^2 u^\ez|^2  [|\nb u^\ez|^2+\ez]^{  \frac{p-\gz}2 }\phi^2 \, e^{-h}d{\rm vol}_g \notag \\
 & \le  - \int_U  Ric_h^N (\nb u^\ez, \nb u^\ez)[|\nb u^\ez|^2+\ez]^{\frac{p-\gz}2 } \phi^2 \, e^{-h}d{\rm vol}_g  \notag  \\
 & \quad + C(n,N,p,\gz,\eta)\int_{U }        \lf ( [|\nb u^\ez|^2+\ez]^{\frac{p-\gz}2 +1} |\nb \phi|^2 + [|\nb u^\ez|^2+\ez]^{\frac{p-\gz}2+2}\phi^2 \lr ) \, e^{-h}d{\rm vol}_g
\end{align}
\end{cor}

Under the Bakry-\'{E}mery  curvature-dimension assumption, we have the following uniform  upper bound.
\begin{lem}  \label{fi} {\color{black}Let $u^\ez$ be the solution to \eqref{plap1}.}
If $\gz<3+\frac{p-1}{N-1}$ and  ${Ric}_h^N\ge -\kz$, then one has
\begin{align}\label{e3.x12}
 & \int_U  |\nb [[|\nb u^\ez|^2+\ez]^{\frac{p-\gz}4} \nb u^\ez]|^2 \phi^2 \, e^{-h}d{\rm vol}_g  \notag \\
 & \le  C(n,N,p,\gz)   \int_U  \kz |\nb u^\ez|^2[|\nb u^\ez|^2+\ez]^{\frac{p-\gz}2 } \phi^2 \, e^{-h}d{\rm vol}_g  \notag  \\
 & \quad +  C(n,N,p,\gz) \int_{U }        \lf ( [|\nb u^\ez|^2+\ez]^{\frac{p-\gz}2+1 } |\nb \phi|^2 +  [|\nb u^\ez|^2+\ez]^{\frac{p-\gz}2+2}\phi^2 \lr ) \, e^{-h}d{\rm vol}_g.
\end{align}

\end{lem}

\begin{proof}

By $Ric_h^N\ge -\kz$ we know that
$$ -   Ric_h^N(\nb u^\ez, \nb u^\ez) \le \kz|\nb u^\ez|^2$$
Thus the first term in the  right-hand side of \eqref{e3.x8} is bounded from above by
$$\kz \int_U  |\nb u^\ez|^2[|\nb u^\ez|^2+\ez]^{\frac{p-\gz}2 } \phi^2 \, e^{-h}d{\rm vol}_g.$$

 On the other hand, a direct calculation leads to
   \begin{align*}
   &|\nb [[|\nb u^\ez|^2+\ez]^{\frac{p-\gz}4} \nb u^\ez]|^2\\&\quad=
[|\nb u^\ez|^2+\ez]^{\frac{p-\gz}2}\left| \nb^2u^\ez +   \frac{p-\gz}2 \frac{\nb u^\ez \otimes  \nb^2u^\ez \nb u^\ez}{|\nb u^\ez|^2+\ez}\right|^2\\
&\quad=[|\nb u^\ez|^2+\ez]^{\frac{p-\gz}2}[| \nb^2u^\ez|^2+ (p-\gz)\frac{|\nb^2u^\ez \nb u^\ez|^2}{|\nb u^\ez|^2+\ez}
+ \frac{(p-\gz)^2}4 \frac{|\nb u^\ez|^2 | \nb^2u^\ez \nb u^\ez|^2}{[|\nb u^\ez|^2+\ez]^2}]\\
&\quad\le C(n,p,\gz)[|\nb u^\ez|^2+\ez]^{\frac{p-\gz}2} | \nb^2u^\ez|^2.
\end{align*}
Thus, up to a constant multiplier, the {\color{black}left-hand} side of \eqref{e3.x8} is bounded by
$$\int_U|\nb [[|\nb u^\ez|^2+\ez]^{\frac{p-\gz}4} \nb u^\ez]|^2e^{-h}d{\rm vol}_g.$$
We therefore conclude  \eqref{e3.x12} from \eqref{e3.x8}.
\end{proof}

Now we are able to prove Theorem \ref{thm2}.
 \begin{proof}[Proof of Theorem \ref{thm2}.]
 Let ${\color{black}w} \in W_{\loc}^{1,p}(\Omega)$ be any {\color{black}positive weighted} $p$-harmonic function in the domain $\Omega$ and {\color{black}$u=-(p-1)\ln w$}. Given any smooth domain $U \Subset \Omega$, for each $\ez \in (0, 1]$, let ${\color{black}u^\ez} \in  C^\fz(U)$ be the solution to \eqref{plap1}. By Lemma \ref{d1983}, we know that $u^\ez \to u\in C^{1,\az}(U)$,
for some $\az \in(0,1)$ uniformly in $\ez>0$ as $\ez\to0$. Using this and choosing suitable test functions $\phi \in C_c^\fz(U)$ in \eqref{e3.x12}, one concludes $[|\nb u^\ez|^2+\ez]^{\frac{p-\gz}4} \nb u^\ez \in W_{\loc}^{1,2} (U)$ uniformly in $\ez \in (0, 1]$.

Next, we claim that
\begin{equation}\label{cl1}
  |\nb u|^{\frac{p-\gz}2} \nb u \in W_{\loc}^{1,2} (U),
\end{equation}
and
\begin{equation}\label{cl3}
  \mbox{ $\nb([|\nb u^\ez|^2+\ez]^{\frac{p-\gz}4} \nb u^\ez) \to \nb(|\nb u|^{\frac{p-\gz}2} \nb u)$ weakly in $L^2_{\loc} (U, \rr^{n\times n})$ as $\ez \to 0$.}
\end{equation}
To see this, for any subdomain $V \Subset U$, by Lemma \ref{fi}, we already have
$$\sup_{\ez \in (0,1]} \|\nb([|\nb u^\ez|^2+\ez]^{\frac{p-\gz}4} \nb u^\ez) \|_{L^2(V,\rr^{n\times n})} < C(\kz,n,N,p,\gz,{\color{black}V}).$$ For any subsequence $\{\ez_j\}_{j\in \nn}$ which converges to $0$, by the weak
compactness of $W^{2,2}(V )$, up to some subsequence one has $\nb([|\nb u^{\ez_j}|^2+\ez_j]^{\frac{p-\gz}4} \nb u^{\ez_j}) \to z$ weakly in $L^2(V, \rr^{n\times n})$
for some function ${\color{black}z \in L^2(V, \rr^{n\times n})}$.
Let $\{e_1, \cdots, e_n\} \subset T_xU$ be a local orthonormal frame at each $x \in U$. Notice that the $n\times n$ matrix
$$\nb([|\nb u^{\ez_j}|^2+\ez_j]^{\frac{p-\gz}4} \nb u^{\ez_j})= \lf (\nb_{e_l}([|\nb u^{\ez_j}|^2+\ez_j]^{\frac{p-\gz}4} \nb_{e_k} u^{\ez_j}) \lr)_{1 \le k,l \le n}.$$
{\color{black}Recalling from Lemma \ref{d1983} that $\nb u^\ez \to \nb u$ in $C^\az(U)$ and $V \Subset U$, for any $\phi \in C_c^\fz(U)$ with $\phi|_V=1$ and $1 \le k,l \le n$, we have
\begin{align*}
  \lim_{j\to 0} \int_U & \nb_{e_l}([|\nb u^{\ez_j}|^2+\ez_j]^{\frac{p-\gz}4} \nb_{e_k} u^{\ez_j}) \phi \, e^{-h}d{\rm vol}_g \\
  &= -\lim_{j\to 0} \int_U ([|\nb u^{\ez_j}|^2+\ez_j]^{\frac{p-\gz}4} \nb_{e_k} u^{\ez_j})\nb_{e_l} (\phi e^{-h}) \, d{\rm vol}_g  \\
   & = - \int_U (|\nb u|^{\frac{p-\gz}2} \nb_{e_k} u)\nb_{e_l} (\phi e^{-h}) \, d{\rm vol}_g \\
   & =  \int_U \nb_{e_l}(|\nb u|^{\frac{p-\gz}2} \nb_{e_k} u) \phi \, e^{-h}d{\rm vol}_g.
\end{align*}}
This shows that in the distributional sense
$$\nb([|\nb u^{\ez_j}|^2+\ez_j]^{\frac{p-\gz}4} \nb u^{\ez_j}) \to \nb(|\nb u|^{\frac{p-\gz}2} \nb u).$$
Thus $z = \nb(|\nb u|^{\frac{p-\gz}2} \nb u )|_V \in L^2(V, \rr^{n\times n})$ in distributional sense. We therefore have $|\nb u|^{\frac{p-\gz}2} \nb u|_V \in W^{1,2}(V )$, which gives \eqref{cl1}.

Moreover, by the arbitrariness of subsequence $\{ \ez_j\}$, we have $$\nb([|\nb u^\ez|^2+\ez]^{\frac{p-\gz}4} \nb u^\ez) \to \nb(|\nb u|^{\frac{p-\gz}2} \nb u)$$ weakly in $L^2(V, \rr^{n\times n})$ as $\ez \to 0$. Hence by the arbitrariness of $V \Subset U$, \eqref{cl3} holds.


Letting $\ez \to 0 $ in \eqref{e3.x12}  and using the convergence in the above verified claim, we obtain
\begin{align}\label{e3.x11}
&  \int_U  |\nb [|\nb u|^{\frac{p-\gz}2} \nb u]|^2 \phi^2 \, e^{-h}d{\rm vol}_g\nonumber\\
 & \quad\le  C(n,N,p,\gz) \kz   \int_U |\nb u|^{p-\gz+2 } \phi^2 \, e^{-h}d{\rm vol}_g  \notag  \\
 & \quad\quad +  C(n,N,p,\gz) \int_{U }        \lf ( |\nb u|^{p-\gz+2 } |\nb \phi|^2 +  |\nb u|^{p-\gz+4}\phi^2 \lr ) \, e^{-h}d{\rm vol}_g.
\end{align}

  Let  $\phi \in C^{\fz}_c(B_{2r})$, where $B_{4r}\subset U$, such that  $\phi=1$ in $B_r$ and $|\nb \phi| \le \frac{C}{r}$.  Then \eqref{e3.x11} becomes
  \begin{align*}
    \int_{B_r}  |\nb [|\nb u|^{\frac{p-\gz}2} \nb u]|^2 \, e^{-h}d{\rm vol}_g
& \le
 C(n,N,p,\gz) \int_{B_{2r} }       \lf [ (\frac{1}{r^2}+\kz)|\nb u|^{p-\gz+2} +|\nb u|^{p-\gz +4} \lr] \, e^{-h}d{\rm vol}_g.
  \end{align*}
Recalling from \eqref{mest0} the Cheng-Yau type gradient {\color{black}estimate} that
$|\nb u|\le  C(n,N,p) \frac{1+\sqrt \kz r}r$ and noting that $\gz<3+\frac{p-1}{N-1}$ {\color{black}guarantees} $p-\gz+2>0$, we deduce
$$|\nb u|^{p-\gz+2} \le  C(n,N,p,\gz) [\frac{1+\sqrt \kz r}r]^{p-\gz+2}.$$ Together with $\frac{1}{r^2}+\kz \le (\frac{1+\sqrt \kz r}r)^2$,
we conclude
\begin{align*}
    \int_{B_r}  |\nb [|\nb u|^{\frac{p-\gz}2} \nb u]|^2 \, e^{-h}d{\rm vol}_g
& \le
 C(n,N,p,\gz) {\rm vol}_h (B_{2r}) \left[\frac{1+\sqrt \kz r}r\right] ^{p-\gz+4}.
  \end{align*}
Dividing both sides by ${\rm vol}_h (B_{ r})$, noting $
{\rm vol}_h (B_{2r})\le e^{\sqrt\kz r}
{\rm vol}_h (B_{ r})$ from the volume comparison \eqref{vc},
  and recalling  $u= -(p-1)\ln w$,  we conclude  \eqref{mest2}.

Note that \eqref{mest1} is just the special case $\gz=p$ of \eqref{mest2}, where
$p<3+\frac2{N-2}$ {\color{black}guarantees} $p< 3+\frac{p-1}{N-1}$ and hence one  can take $\gz=p$ in \eqref{mest2}.
\end{proof}

Finally, we compare our proof with \cite{wz10,dd16}, in particular, the crucial pointwise lower bound given in Lemma \ref{keylem1} and Lemma \ref{keylem2}.

\begin{rem} \label{finalrmk} \rm
(i) It was {\color{black}well known that a} positive (weighted)  $p$-harmonic function $w$,
and hence   $\ln w$, is always smooth outside of the null set $E_w$ of $\nb \ln w$.
  In $\Omega\setminus E_w$,  the proof of  Lemma \ref{keylem2}  works for $\ln w$
 so to get  \eqref{lbdd-x1} with $u^\ez$ replaced by $\ln w$ and $\ez=0$,
 dividing both sides of which by $|\nb \ln w|^{4}$, for $0<\eta<1/2$ one gets
\begin{align}\label{lbdd-xx1}
     (1+\eta) |\nb^2 \ln w|^2  & \ge
2\frac{  | \nb ^2\ln w \nb \ln w |^2 }{| \nb \ln w |^2}   +
\left(\frac {(p-1)^2}{N-1}  -1\right)  \frac{(\bdz_\fz \ln w)^2}{|\nb \ln w|^4}\nonumber\\
&\quad\quad - (1+\eta) \frac{\la \nb \ln w, \nb h\ra^2}{N-n}
 - C(n,N,p)\frac1\eta |\nb \ln w|^4.
    \end{align}
If $\gz<3+\frac{p-1}{N-1}$, using \eqref{lbdd-xx1} and  noting that {\color{black}the} proof of Lemma \ref{keylem1}  works for $\ln w$,
we get \eqref{lbdd2} with $u^\ez$ replaced by $\ln w$ and $\ez=0$,  that is,  for $\eta>0$ sufficiently small,
\begin{align}\label{lbddx2}
&(1-\eta)|\nb^2\ln w|^2 +   (p-\gz)    \frac{|\nb^2 \ln w \nb \ln w|^2}{|\nb \ln w|^2 }
 + (p-2) (2-\gz)    \frac{(\bdz_\fz \ln w)^2} { |\nb \ln w|^4}\nonumber\\
&\quad\ge  \eta |\nb^2\ln w|^2
  -   \frac{\la \nb \ln w, \nb h\ra^2}{N-n}  - C(n,N,p,\gz)\frac1\eta |\nb \ln w|^4.
    \end{align}
From the proof, we see that both of the coefficient  $2$ of $\frac{  | \nb ^2\ln w \nb \ln w |^2 }{| \nb \ln w |^2}$ and
  the coefficient $ \frac {(p-1)^2}{N-1}  -1 $ of $(\bdz_\fz \ln w)^2 $ in \eqref{lbdd-xx1} are critical to guarantee the existence of sufficiently small $\eta>0$  in \eqref{lbddx2} when $\gz<3+\frac{p-1}{N-1}$.

On the other hand, instead of \eqref{lbdd-xx1},  recall  the following lower bound obtained in
  \cite{dd16} by using Lemma \ref{keylem3} and the equation \eqref{plap3}:
\begin{equation}\label{dd}|\nb^2 \ln w|^2\ge \frac{  | \nb ^2\ln w \nb \ln w |^2 }{| \nb \ln w |^2}
-2\frac{p-1}{n-1}\Delta_\fz \ln w  + \frac1{N-1}|\nb \ln w|^2 -\frac{\langle \nb\ln w,\nb h\rangle^2}{N-n},
\end{equation}
and also, when $N=n$ and $h\equiv 1$, recall the following lower bound  derived  in \cite{wz10}  via Lemma \ref{keylem3} and \eqref{plap3}:
\begin{equation}\label{wz}|\nb^2 \ln w|^2\ge  [1+\min\{\frac{(p-1)^2}{n-1},1\}]\frac{  | \nb ^2\ln w \nb \ln w |^2 }{| \nb \ln w |^2}-2\frac{p-1}{n-1} \bdz_\fz \ln w + \frac1{n-1}|\nb \ln w|^2.\end{equation}
From \eqref{dd} and \eqref{wz}, via a direct check one can {\color{black}conclude $|\nb \ln w|^{\frac{p-\gz+2}2}\in W^{1,2}_\loc$ for  $\gz < 2$}, but  NOT for all   $\gz<3+\frac{p-1}{N-1}$.

(ii) Moreover, unlike \cite{wz10,dd16} where {\color{black}the authors} differentiate the equation \eqref{plap3} for $\ln w$,
we directly derive {\color{black}an} upper bound from Bochner formula for the left-hand side of  \eqref{lbddx2}
with respect to $[|\nb u^\ez|^2+\ez]^{\frac{p-\gz}2}e^{-h}d{\rm vol}_g$.
\end{rem}

\renewcommand{\thesection}{Appendix A}
 \renewcommand{\thesubsection}{ A }
\newtheorem{lemapp}{Lemma \hspace{-0.15cm}}
\newtheorem{thmapp}[lemapp] {Theorem  \hspace{-0.15cm}}
\newtheorem{corapp}[lemapp] {Corollary \hspace{-0.15cm}}
\newtheorem{remapp}[lemapp]  {Remark  \hspace{-0.15cm}}
\newtheorem{defnapp}[lemapp]  {Definition  \hspace{-0.15cm}}
\newtheorem{egapp}[lemapp]  {Example  \hspace{-0.15cm}}
\renewcommand{\theequation}{A.\arabic{equation}}

\renewcommand{\thelemapp}{A.\arabic{lemapp}}

\section{Proof of Lemma \ref{d1983}}

In the appendix, we show Lemma \ref{d1983} by checking equations \eqref{plap3} and \eqref{plap1} are special cases considered in \cite{d83}. To this end, we recall the result in \cite{d83}.

Let $\Omega$ be a domain of $M^n$. Consider the equation
\begin{equation}\label{d1}
   -\ddiv \, \vec{a}(x,\nb u) + b (x,\nb u) =0  \quad \mbox{in $\Omega$}
\end{equation}
where $\vec{a}$ is a map from $\Omega \times \rn$ to $\rn$ and $b$ maps $\Omega \times \rn$ to $\rr$.
Let $\{e_1, \cdots, e_n\} \subset T_x\Omega$ be a local orthonormal frame at each $x \in \Omega$.
By a weak solution of \eqref{d1} we mean a function $u \in W^{1,p}_\loc(\Omega)$ such that
\begin{equation}  \label{defp1}
  \int_\Omega  [\la \vec{a}(x,\nb u), \nb \phi \ra + b (x,\nb u)\phi] \, d{\rm vol}_g=0\quad\forall \phi\in C^\fz_c(\Omega).
  \end{equation}

Assume the following holds for $\vec{a}=(a_1, \cdots, a_n)$ and $b$.
\begin{equation}\label{A1}
  \sum_{i,j=1}^n \frac{\pa a_j}{\pa \eta_i}  (x,\eta) \xi_i\xi_j \ge \gz_0  |\eta|^{p-2}|\xi|^2, \quad  \forall\xi \in \rr^n,p>1,                  \tag{A$_1$}
\end{equation}
\begin{equation}\label{A2}
  \lf|\frac{\pa a_j}{\pa \eta_i} \lr| \le \gz_1  |\eta|^{p-2}, \quad   1 \le i,j \le n,                 \tag{A$_2$}
\end{equation}
\begin{equation}\label{A3}
  |\nb_{e_i} a_j(x,\eta)| \le \gz_1  |\eta|^{p-1}, \quad  1 \le i,j \le n,                 \tag{A$_3$}
\end{equation}
\begin{equation}\label{A4}
 | b(x,\eta)| \le \gz_1  |\eta|^{p},                 \tag{A$_4$}
\end{equation}
and
\begin{equation}\label{AB}
  \ |\nb_{e_i} b(x,\eta)| \le \gz_1  |\eta|^{p}, \  \lf|\frac{\pa b}{\pa \eta_i}(x,\eta) \lr| \le \gz_1  |\eta|^{p-1}, \quad  1 \le i \le n,                 \tag{B}
\end{equation}
for all $\eta \in \rr^n$, where $\gz_i$ are positive constants, $i=0,1$.

For any smooth domain $U \Subset \Omega$ and $\ez \in (0,1]$, consider the regularized equation
\begin{equation}\label{d2}
   -\ddiv \, \vec{a^\ez}(x,\nb u^\ez) + b^\ez (x,\nb u^\ez) =0  \quad \mbox{in $U$ and $u^\ez=u$ on $\pa U$}
\end{equation}
where $\vec{a^\ez}$ is a map from $U \times \rn$ to $\rn$ and $b^\ez$ maps $U \times \rn$ to $\rr$ such that
$$  \lim_{\ez \to 0}\vec{a^\ez}(x,\eta) = \vec{a}(x,\eta) \mbox{ and } \lim_{\ez \to 0}b^\ez(x,\eta) = b(x,\eta) \quad \forall (x,\eta)\in \Omega \times \rn. $$
The weak solution of \eqref{d2} is defined similarly as \eqref{defp1}.
Assume the following holds for $\vec{a^\ez}=(a_1^\ez, \cdots, a_n^\ez)$ and $b^\ez$.
\begin{equation}\label{A1e}
  \sum_{i,j=1}^n \frac{\pa a_j^\ez}{\pa \eta_i}  (x,\eta) \xi_i\xi_j \ge \gz_0  (\ez+|\eta|^2)^{\frac{p-2}2}|\xi|^2, \quad  \xi \in \rr^n,p>1,                  \tag{A$_{1,\ez}$}
\end{equation}
\begin{equation}\label{A2e}
  \lf|\frac{\pa a_j^\ez}{\pa \eta_i} \lr| \le \gz_1  (\ez+|\eta|^2)^{\frac{p-2}2}, \quad   1 \le i,j \le n,                 \tag{A$_{2,\ez}$}
\end{equation}
\begin{equation}\label{A3e}
  |\nb_{e_i} a_j^\ez(x,\eta)| \le \gz_1  (\ez+|\eta|^2)^{\frac{p-1}2}, \quad  1 \le i,j \le n,                 \tag{A$_{3,\ez}$}
\end{equation}
\begin{equation}\label{A4e}
 | b^\ez(x,\eta)| \le \gz_1  (\ez+|\eta|^2)^{\frac{p}2},           \tag{A$_{4,\ez}$}
\end{equation}
for all $\eta \in \rr^n\setminus \{0\}$.

We recall the results in \cite{d83} as follows.
\begin{thmapp}\label{di83}
   Let $\ez \in (0,1]$ and $U\Subset \Omega$. Assume \eqref{A1}-\eqref{A4}, \eqref{AB} and \eqref{A1e}-\eqref{A4e} hold. Then there exists a unique solution  $u^\ez\in C^\fz(U )\cap C^0(\overline U )$ to \eqref{d2}, and moreover, $u^\ez\to u$ in $C^{0}(\overline U)$ and $u^\ez\to u$ in $C^{1,\az}(V)$ uniformly in $\ez>0$
  as $\ez\to 0$ for all $V \Subset U$ where $u$ is the solution to \eqref{d1}.
   As a consequence, $u \in C^{1,\az}(\Omega)$.
\end{thmapp}
Theorem \ref{di83} is a combination of Theorem 1 and Theorem 2 in \cite{d83} and several intermediate results in the proof of these two theorems in \cite{d83}. Indeed, the existence, uniqueness and $C^\fz$-regularity of $u^\ez$ is by elliptic theory in PDE; see for example \cite{gt77}.
Based on these facts, in \cite{d83},
the author first showed that under \eqref{A1}-\eqref{A4}, \eqref{AB} and \eqref{A1e}-\eqref{A4e}, $u^\ez\to u$ in $W^{1,p}(U)$ uniformly in $\ez>0$ in section 2. Moreover, $\|u^\ez\|_{L^\fz(U)} \le \max_{x \in \pa U} \{|u(x)|\}$. Thus recalling that $u^\ez|_{\pa U} = u|_{\pa U}$, we know $u^\ez\to u$ in $C^{0}(\overline U)$. See the discussion around (2.7) in \cite{d83}. Then the author showed that
$\|u^\ez\|_{C^{1,\az}(V)}$ is uniformly bounded independently of $\ez \in (0,1]$ and finally showed that $u^\ez \to u$ in $C^{1,\az}(V)$ and $u \in C^{1,\az}(U)$ for all $V \Subset U$. By the arbitrariness of $U\Subset \Omega$, one has $u \in C^{1,\az}(\Omega)$.

\begin{proof}[Proof of Lemma \ref{d1983}] It suffices to check equations \eqref{plap3} and \eqref{approx} are special ones of \eqref{d1} and \eqref{d2} respectively. To this end,
let $\vec{a}(x,\eta)=e^{-h(x)} |\eta|^{p-2}\eta$, $b(x,\eta)=-e^{-h(x)} |\eta|^{p}$, $\vec{a^\ez}(x,\eta)=e^{-h(x)} (|\eta|^2+\ez)^{\frac{p-2}2}\eta$, and $b^\ez(x,\eta)=-e^{-h(x)} (|\eta|^2+\ez)^{\frac{p-2}2}|\eta|^{2}$ for all $x \in U$ and $\eta \in\rr^n$.
Then in the weak sense, the equations
$$\int_\Omega  [\la \vec{a}(x,\nb u), \nb \phi \ra + b (x,\nb u)\phi] \, d{\rm vol}_g=0, \quad\forall \phi\in C^\fz_c(\Omega) $$
and
$$ \int_\Omega  [\la \vec{a^\ez}(x,\nb u), \nb \phi \ra + b^\ez (x,\nb u)\phi] \, d{\rm vol}_g=0,\quad\forall \phi\in C^\fz_c(\Omega)$$
are exactly \eqref{plap3} and \eqref{approx} respectively.

We show $\vec{a}$ satisfies \eqref{A1}. Noting that $a_j(x,\eta)=e^{-h(x)} |\eta|^{p-2}\eta_j$, we compute
$$\frac{\pa a_j}{\pa \eta_i}(x,\eta)= e^{-h(x)}[(p-2)|\eta|^{p-4}\eta_i\eta_j + \dz_{ij}|\eta|^{p-2}], \quad \forall 1\le i,j \le n$$
where $\dz_{ij}=1$ if $i=j$ and $\dz_{ij}=0$ if $i\ne j$.
Thus
\begin{align*}
  \sum_{i,j=1}^n \frac{\pa a_j}{\pa \eta_i}  (x,\eta) \xi_i\xi_j & = e^{-h(x)}\sum_{i,j=1}^n[((p-2)|\eta|^{p-4}\eta_i\eta_j + \dz_{ij}|\eta|^{p-2})\xi_i\xi_j]\\
   & = e^{-h(x)}|\eta|^{p-4}[(p-2)(\sum_{i=1}^n \eta_i\xi_i)^2 + |\eta|^2|\xi|^2], \quad  \forall\xi \in \rr^n.
\end{align*}
If $1<p<2$, we have
\begin{align*}
  \sum_{i,j=1}^n \frac{\pa a_j}{\pa \eta_i}  (x,\eta) \xi_i\xi_j & \ge
    e^{-h(x)}(p-1)|\eta|^{p-2}|\xi|^2, \quad  \forall \xi \in \rr^n.
\end{align*}
And if $p \ge 2$, we have
\begin{align*}
  \sum_{i,j=1}^n \frac{\pa a_j}{\pa \eta_i}  (x,\eta) \xi_i\xi_j & \ge
    e^{-h(x)}|\eta|^{p-2}|\xi|^2, \quad  \forall \xi \in \rr^n.
\end{align*}
By taking $\gz_0 := \min_{x \in \overline U}\{e^{-h(x)}\}$,
we conclude that $a$ satisfies \eqref{A1}.
By direct computations, one can also check $\vec{a},\vec{a^\ez} \in C^\fz(U \times \rr^n, \rr^n)$, $b,b^\ez\in C^\fz(U \times \rr^n)$ satisfy \eqref{A2}-\eqref{A4}, \eqref{AB} and \eqref{A1e}-\eqref{A4e} respectively. We omit the details.
Thus by Theorem \ref{di83}, we get the desired result.
\end{proof}

\bigskip
\noindent
\textbf{Acknowledgement.} The authors would like to thank the anonymous referees for the careful reading, valuable comments, and many detailed corrections which improve the final presentation of the paper significantly. In particular, we are grateful to the referees for pointing out references \cite{ags15,eks15} that enhance the precision of the literature review in the introduction.

\noindent Jiayin Liu

\noindent Department of Mathematics and Statistics, University of Jyv\"{a}skyl\"{a}, P.O. Box 35
(MaD), FI-40014, Jyv\"{a}skyl\"{a}, Finland

\noindent{\it E-mail }:  \texttt{jiayin.mat.liu@jyu.fi}

\bigskip

\noindent Shijin Zhang

\noindent School of Mathematical Science, Beihang University, Changping District Shahe Higher Education Park
  South Third Street No. 9, Beijing 102206, P. R. China

\noindent{\it E-mail }:  \texttt{shijinzhang@buaa.edu.cn}

\bigskip

\noindent  Yuan Zhou

\noindent
School of Mathematical Science, Beijing Normal University, Haidian District Xinjiekou Waidajie No.19, Beijing 10875, P. R. China

\noindent{\it E-mail }:  \texttt{yuan.zhou@bnu.edu.cn}

\end{document}